\documentclass[
]{amsart}
\usepackage[latin1]{inputenc}
\usepackage[boxed]{algorithm2e}
\usepackage{amsmath}\usepackage{amsthm}\usepackage{amssymb}\usepackage{mathrsfs}\usepackage{amscd}\usepackage{graphicx}\usepackage{subfigure}\usepackage{amsfonts}\usepackage{amsxtra}\usepackage{color}
\usepackage{multirow}
\usepackage{amscd}

\DeclareMathOperator{\vol}{vol}
\DeclareMathOperator{\PFP}{PFP}
\DeclareMathOperator{\FP}{FP}

\DeclareMathOperator{\supp}{supp}
\theoremstyle{definition}\newtheorem{definition}{Definition}[section]\newtheorem{example}[definition]{Example}\newtheorem{remark}[definition]{Remark}
\theoremstyle{plain}\newtheorem{theorem}[definition]{Theorem}\newtheorem{lemma}[definition]{Lemma}\newtheorem{corollary}[definition]{Corollary}\newtheorem{proposition}[definition]{Proposition}

\newcommand{\R}{\mathbb{R}}\newcommand{\C}{\mathbb{C}}

\begin{document}
\title{Random Tight Frames}

%

\author{M.~Ehler}
\date{}

\address{National Institutes of Health, National Institute of Child Health and Human Development, Section on Medical Biophysics, Bethesda, MD 20892}

\email{ehlermar@mail.nih.gov}

\address{University of Maryland, Department of Mathematics, Norbert Wiener Center, College Park, MD 20742}
\email{ehlermar@math.umd.edu}


\keywords{frames, probability, optimal configurations}

\subjclass[2010]{42C15}

\begin{abstract}
We introduce probabilistic frames to study finite frames whose elements are chosen at random. While finite tight frames generalize orthonormal bases by allowing redundancy, independent, uniformly distributed points on the sphere approximately form a finite unit norm tight frame (FUNTF). In the present paper, we develop probabilistic versions of tight frames and FUNTFs to significantly weaken the requirements on the random choice of points to obtain an approximate finite tight frame. Namely, points can be chosen from any probabilistic tight frame, they do not have to be identically distributed, nor have unit norm. We also observe that classes of random matrices used in compressed sensing are induced by probabilistic tight frames. 
\end{abstract}

\maketitle








\section{Introduction}
Frames are basis-like systems that span a vector space but allow for linear dependency, which can be used to reduce noise, find sparse representations, or obtain other desirable features unavailable with orthonormal bases. They have proven useful in fields like spherical codes, compressed sensing, signal processing, and wavelet analysis \cite{Casazza:2003aa,Christensen:2003aa,Daubechies:1986aa,Ehler:2007aa,Ehler:ab,Ehler:2010ae,Ehler:2008ab,Ehler:aa,Goyal:2001aa,Feichtinger:2003aa,Grochenig:2001aa,Shen:2006aa}. Tight frames even provide a Parseval type formula similar to orthonormal bases. However, characterizations and constructions of finite tight frames and finite unit norm tight frames (FUNTFs) were needed \cite{Casazza:2003aa}. A general characterization of all FUNTFs was given by Benedetto and Fickus in \cite{Benedetto:2003aa}, where they proved that the FUNTFs are exactly the minimizers of a functional called the frame potential.  This was extended to finite tight frames in \cite{Waldron:2003aa}. Casazza and Fickus have considered the frame potential in the framework of fusion frames \cite{Casazza:2009aa}. To approximate a FUNTF, Goyal, Vetterli, and Thao considered in \cite{Goyal:1998aa} $n$ random points on the sphere. In fact, they showed that  independent, identically distributed (i.i.d.) points according to the uniform distribution on the sphere asymptotically (as $n\rightarrow \infty$) become a FUNTF.

The present paper is concerned with frames in a probabilistic setting and the generalization of the results of Goyal, Vetterli, and Thao. Our aim is to allow for a more flexible choice of $n$ points while still preserving the asymptotical tight frame property. We first introduce probabilistic frames and adopt many concepts and properties from finite frames  to the probabilistic setting. Probabilistic versions of frames, tight frames, Parseval frames, and FUNTFs are developed. After observing that the uniform distribution on the sphere is a probabilistic unit norm tight frame, we extend the results about the random choice of $n$ points on the sphere as follows: in comparison to \cite{Goyal:1998aa}, we are not limited to the uniform distribution and allow for any probabilistic tight frame. Moreover, the points do not have to be identically distributed nor must they lie on a sphere. This means a significant weakening of the assumptions in \cite{Goyal:1998aa} and offers much more flexibility. We use this extension to observe that Bernoulli, Gaussian, and sub-Gaussian random matrices, which are used in compressed sensing, fit into this scheme by choosing their rows according to probabilistic tight frames.

To better understand probabilistic tight frames, we minimize the frame potential as introduced by Benedetto and Fickus within a probabilistic setting. In fact, we characterize probabilistic tight frames as minimizers of the probabilistic frame potential, which also generalizes \cite{Waldron:2003aa}. Relations to spherical $t$-designs \cite{Delsarte:1977aa,Seidel:2001aa} are also discussed.

The outline is as follows: In Section \ref{section:frame potential}, we recall finite frames, the frame potential, and the characterization of its minimizers as derived by Benedetto and Fickus. We also recall the results of Goyal, Vetterli, and Thao about the random choice of $n$ points on the sphere. Section \ref{section:characterization} is dedicated to studying probabilistic frames that are introduced in Section \ref{section:probabilistic unit norm tight frames}. Well-known properties from finite frames are adopted to the probabilistic setting, and we define and study probabilistic tight frames. We then generalize the results of Goyal, Vetterli, and Thao in Section \ref{section:sampling}. In Section \ref{section:random points}, we study the probabilistic frame potential. We show in Section \ref{section:introducing PFP} that its minimizers are the probabilistic tight frames, and the relations to spherical $t$-designs are addressed in Section \ref{section:spherical 2-designs}. Conclusions are given in Section \ref{section:conclusions}.


\section{Background}\label{section:frame potential}
A collection of points $\{x_i\}_{i=1}^n\subset\R^d$ is called a \emph{finite frame for $\R^d$} if there are two constants $0<A\leq B$ such that
\begin{equation}\label{eq:frame ineq}
A\|x\|^2 \leq \sum_{i=1}^n |\langle x,x_i\rangle|^2 \leq B\|x\|^2,\quad\text{for all $x\in\R^d$.}
\end{equation}
The constants $A$ and $B$ are called \emph{lower and upper frame bounds}, respectively. In fact, finite frames are the finite spanning sets \cite{Christensen:2003aa}: 
\begin{lemma}\label{lemma:frame if and only if spans}
The sequence $\{x_i\}_{i=1}^n\subset\R^d$ is a finite frame for $\R^d$ if and only if it spans $\R^d$.
\end{lemma}
The frame property can also be expressed by means of operators. Given a collection of $n$ points $\{x_i\}_{i=1}^n$ in $\R^d$, we call 
 \begin{equation*}
 F: \R^d \rightarrow \R^n, \quad x\mapsto \big(\langle x,x_i \rangle\big)_{i=1}^n
 \end{equation*}
the \emph{analysis operator}. Its adjoint operator 
 \begin{equation*}
  F^*: \R^n \rightarrow\R^d , \quad (c_i)_{i=1}^n\mapsto \sum_{i=1}^n c_i x_i
 \end{equation*}
is called the \emph{synthesis operator}. If the collection $\{x_i\}_{i=1}^n$ is a finite frame for $\R^d$, then the \emph{frame operator} $S=F^* F$ is positive, self-adjoint, and invertible \cite{Christensen:2003aa}. In this case, the following reconstruction formula holds,
\begin{equation}\label{eq:reconstruction for frames}
x = \sum_{j=1}^n  \langle S^{-1} x_i,x\rangle x_i =  \sum_{j=1}^n  \langle  x_i,x\rangle S^{-1}x_i ,\text{ for all $x\in\R^d$,}
\end{equation}
and $\{S^{-1} x_i\}_{i=1}^n$, in fact, is a frame too, called the \emph{canonical dual frame}.

Frames, whose lower and upper frame bounds coincide, play a special role, and we call a collection of points $\{x_i\}_{i=1}^n\subset\R^d$ a \emph{finite tight frame for $\R^d$} if there is a positive constant $A$ such that 
\begin{equation}\label{eq:tightness}
A\|x\|^2 = \sum_{i=1}^n |\langle x,x_i\rangle|^2,\quad\text{for all $x\in\R^d$.}
\end{equation}
The constant $A$ is called the \emph{tight frame bound}. Note that every finite tight frame gives rise to the expansion
\begin{equation}\label{eq:tight expansion}
x = \frac{1}{A}\sum_{i=1}^n \langle x,x_i\rangle x_i,\quad\text{for all $x\in\R^d$.}
\end{equation}
In this sense they are a generalization of orthonormal bases. The following lemma summarizes the standard characterizations of tight frames, cf.~\cite{Christensen:2003aa}:
\begin{lemma}\label{lemma:equivalent tight frame conditions discrete}
Let $\{x_i\}_{i=1}^n$ be a collection of vectors in $\R^d$, and let $A$ be a positive constant. The following points are equivalent:
\begin{itemize}
\item[\textnormal{(i)}]  $\{x_i\}_{i=1}^n$ is a finite tight frame for $\R^d$ with frame bound $A$,
\item[\textnormal{(ii)}] $F^* F=A\mathcal{I}_d$,
\item[\textnormal{(iii)}] Equation \eqref{eq:tight expansion} holds. 
\end{itemize}
\end{lemma}

If $A=1$ in \eqref{eq:tightness}, then we call $\{x_i\}_{i=1}^n\subset\R^d$ a \emph{finite Parseval frame}. If all elements of a finite tight frame have unit norm, we call them a \emph{finite unit norm tight frame (FUNTF)} for $\R^d$. Note that a FUNTF that is also Parseval must be an orthonormal basis \cite{Christensen:2003aa}. In fact, the frame bounds of a FUNTF are given by:
\begin{lemma}[\cite{Goyal:1998aa}]\label{lemma:FUNTF frame bound}
If $\{x_i\}_{i=1}^n\subset \R^d$ is a FUNTF, then the frame bound $A$ equals $n/d$. 
\end{lemma}
Every finite frame for $\R^d$ gives rise to a Parseval frame, cf.~\cite{Christensen:2003aa}:
\begin{lemma}\label{lemma:Parseval S1/2}
If $\{x_i\}_{i=1}^n$ is a finite frame for $\R^d$ with frame operator $S$, then $\{S^{-1/2}x_i\}_{i=1}^n$ is a finite Parseval frame for $\R^d$.
\end{lemma}

The following identity and inequality for Parseval frames have been derived in
\cite{Balan:2007aa}:
\begin{theorem}[\cite{Balan:2007aa}]\label{theorem:fundamental identity}
Let $\{x_i\}_{i=1}^n\subset\R^d$ be a finite Parseval frame for $\R^d$. For every subset $J\subset \{1,\ldots,n\}=:\mathcal{N}_n$ and every $x\in\R^d$, we have
\begin{align*}
\sum_{i\in J} |\langle x,x_i\rangle |^2 - \big\| \sum_{i\in J}\langle x,x_i\rangle x_i     \big\|^2 & = \sum_{i\in \mathcal{N}_n\setminus J} |\langle x,x_i\rangle |^2 - \big\| \sum_{i\in \mathcal{N}_n\setminus J}\langle x,x_i\rangle x_i     \big\|^2,\\
\sum_{i\in J} |\langle x,x_i\rangle |^2 - \big\| \sum_{i\in \mathcal{N}_n\setminus J}\langle x,x_i\rangle x_i     \big\|^2 & \geq \frac{3}{4}\|x\|^2.
\end{align*}
\end{theorem}

Given $n$ points $\{x_i\}_{i=1}^n$ on the sphere $S^{d-1}=\{x\in\R^d:\|x\|=1\}$, the frame potential as introduced by Benedetto and Fickus in \cite{Benedetto:2003aa} is 
\begin{equation}\label{eq:frame potential}
\FP(\{x_i\}_{i=1}^n) = \sum_{i=1}^n\sum_{j=1}^n |\langle x_i,x_j\rangle |^2.
\end{equation}
For fixed $n$, they characterized its minimizers:  
\begin{theorem}[\cite{Benedetto:2003aa}]\label{theorem:Benedetto Fickus}
Let $n$ be fixed and consider the minimization of the frame potential among all collections of $n$ points on the sphere $S^{d-1}$.
\begin{itemize}
\item[$n\leq d$:] The minimum of the frame potential is $n$. The minimizers are exactly the orthonormal systems for $\R^d$ with $n$ elements.
\item[$n\geq d$:]
The minimum of the frame potential is $\frac{n^2}{d}$. The minimizers are exactly the FUNTFs for $\R^d$ with $n$ elements.
\end{itemize}
\end{theorem}
The overlap $n=d$ in Theorem \ref{theorem:Benedetto Fickus} is not a problem since every FUNTF with $n=d$ elements is an orthonormal basis. Waldron derived an estimate of the frame potential for general points in $\R^d$, not necessarily on the sphere:
\begin{theorem}[\cite{Waldron:2003aa}]\label{theorem:Waldron}
If $\{x_i\}_{i=1}^n\subset\R^d$ are not all zero and $n\geq d$, then 
\begin{equation*}
\frac{\sum_{i=1}^n\sum_{j=1}^n |\langle x_i,x_j\rangle |^2 }{\big(\sum_{i=1}^n \|x_i\|^2 \big)^2}\geq \frac{1}{d},
\end{equation*}
and equality holds if and only if $\{x_i\}_{i=1}^n$ is a finite tight frame for $\R^d$.
\end{theorem} 
Goyal, Vetterli, and Thao have shown in \cite{Goyal:1998aa} that independent and uniformly distributed points on the sphere converge towards a FUNTF. To properly formulate the convergence, let $\mathcal{M}(\mathcal{B},S^{d-1})$ denote the collection of probability measures on $S^{d-1}$ with respect to the induced Borel $\sigma$ algebra $\mathcal{B}$. If $Z:S^{d-1}\rightarrow U\subset\R^{p\times q}$ is a random matrix/vector, distributed according to $\mu\in\mathcal{M}(\mathcal{B},S^{d-1})$, then we simply write $Z\in U$ for notational convenience. The expectation of $Z$ is defined by $E(Z):=\int_{S^{d-1}} Z(x) d\mu(x)$, where the integral is taken component-wise. Note that for a collection of random vectors $\{X_i\}_{i=1}^n\subset S^{d-1}$, the frame operator is a random matrix. 
\begin{theorem}[\cite{Goyal:1998aa}]\label{theorem:Goyal}
For any $n$, let $\{X_{k,n}\}_{k=1}^n\subset S^{d-1}$ be a collection of $n$ random vectors, i.i.d.~according to the uniform probability distribution on the sphere. If $F_n$ denotes the random matrix associated to the analysis operator of $\{X_{k,n}\}_{k=1}^n$, then the matrix operator $\frac{1}{n}F^*_n F_n$ converges towards $\frac{1}{d}\mathcal{I}_d$ in the mean squared sense, i.e., $E(\|\frac{1}{n}F_n^* F_n-\frac{1}{d}\mathcal{I}_d\|_{\mathcal{F}}^2)\rightarrow 0$, where $\|\cdot\|_\mathcal{F}$ denotes the Frobenius norm.
\end{theorem}
Note that $\frac{1}{n}F^*_n F_n=\frac{1}{d}\mathcal{I}_d$ would mean that we have a FUNTF, cf.~Lemma \ref{lemma:equivalent tight frame conditions discrete}. In the present paper, we develop a framework that leads to a significant generalization of Theorem \ref{theorem:Goyal}.


\section{Probabilistic Frames}\label{section:characterization}

\subsection{Probabilistic Tight Frames}\label{section:probabilistic unit norm tight frames}
In this section, we shall introduce a probabilistic analogue of finite frames. Let $K$ be a nonempty subset of $\R^d$ and let $\mathcal{M}(\mathcal{B},K)$ denote the collection of probability measures on $K$ with respect to the induced Borel $\sigma$ algebra $\mathcal{B}$. 
\begin{definition}\label{def:prob frame 0}
A probability measure $\mu\in\mathcal{M}(\mathcal{B},K)$ is called a \emph{probabilistic frame for $\R^d$} if there are constants $0<A\leq B$ such that
 \begin{equation*}
 A\|x\|^2 \leq \int_{K} |\langle x,y\rangle |^2 d\mu (y) \leq B\|x\|^2,\quad\text{for all $x\in\R^d$.}
 \end{equation*}
The constants $A$ and $B$ are called \emph{lower and upper probabilistic frame bounds}, respectively. If only the upper inequality holds, then we call $\mu$ a \emph{Bessel measure}. A probabilistic frame $\mu$ for $\R^d$ is called a \emph{probabilistic unit norm frame} if $K=S^{d-1}$.
\end{definition}
It should be mentioned that Definition \ref{def:prob frame 0} is not entirely new, but constitutes a shift of perspective:
\begin{remark}
In standard continuous frame theory, the measure $\mu$ is fixed and elements in a Hilbert space form the frame that is indexed by a continuous set. Definition \ref{def:prob tight frame} means a shift of perspective because we identify the index set with the elements in the Hilbert space and hold them fixed (to be $K$). We now allow the measure $\mu$ to vary, which then encodes the frame. 
\end{remark}
If $\{x_i\}_{i=1}^n$ is a frame for $\R^d$, then the normalized counting measure $\frac{1}{n}\mu_{x_1\ldots,x_n}$ is a probabilistic frame for $\R^d$ with respect to any subset $K$ that contains $\{x_i\}_{i=1}^n$. Thus, Definition \ref{def:prob frame 0} extends the concept of finite frames for $\R^d$. 

The support of $\mu\in\mathcal{M}(\mathcal{B}, K )$ is
\begin{equation*}
\supp(\mu)=\{x\in  K  : \mu(U_x)>0, \text{ for all open subsets $U_x\subset  K $ that contain $x$}  \},
\end{equation*}
and the following is the probabilistic counterpart of Lemma \ref{lemma:frame if and only if spans}:
\begin{proposition}
Assume that $K\subset \R^d$ is bounded. A probability measure $\mu\in\mathcal{M}(\mathcal{B}, K )$ is a probabilistic frame for $\R^d$ if and only if its support spans $\R^d$. 
\end{proposition}
\begin{proof}
If the support does not span $\R^d$, then there exists an element $x\in\supp(\mu)^\bot$ that satisfies $ \int_{ K } |\langle x,y\rangle |^2 d\mu (y) =0$. Therefore, $\mu$ cannot be a probabilistic frame.  

For the reverse implication, we observe that the Cauchy-Schwartz inequality yields
 \begin{equation}\label{eq:Bessel}
 \int_{K} |\langle x,y\rangle |^2 d\mu (y) \leq \sup_{y\in K}(\|y\|^2) \|x\|^2,\quad\text{for all $x\in\R^d$.}
 \end{equation}
Since $K$ is bounded, $\mu$ is a Bessel measure and the upper probabilistic frame bound $B$ exists. To find a lower probabilistic frame bound, let us define
\begin{equation*}
A:=\inf_{x\in\R^d} \big(\frac{\int_{ K } |\langle x,y\rangle |^2 d\mu (y)}{\|x\|^2} \big) = \inf_{x\in S^{d-1}} \big(\int_{ K } |\langle x,y\rangle |^2 d\mu (y)\big).
\end{equation*}
Due to the dominated convergence theorem, the mapping $x\mapsto \int_{ K } |\langle x,y\rangle |^2 d\mu (y)$ is continuous and the infimum is in fact a minimum since $S^{d-1}$ is compact. Let $x$ be in $S^{d-1}$ such that 
\begin{equation*}
A=\int_{ K } |\langle x,y\rangle |^2 d\mu (y).
\end{equation*} 
Since $\supp(\mu)$ spans $\R^d$, $x$ cannot be in the orthogonal complement of $\supp(\mu)$, and thus there is $y_0\in\supp(\mu)$ such that $|\langle x,y_0\rangle|^2>0$. Therefore, there is $\varepsilon >0$ and an open subset $U_{y_0}\subset  K $ satisfying $y_0\in U_{y_0}$ and $|\langle x,y\rangle|^2>\varepsilon$, for all $y\in U_{y_0}$. Since $\mu(U_{y_0})>0$, we obtain $A\geq \varepsilon \mu(U_{y_0})>0$, which concludes the proof.
\end{proof}
The \emph{analysis operator} 
 \begin{equation*}
 F: \R^d \rightarrow L_2(K,\mu), \quad x\mapsto \langle x,\cdot \rangle_{\R^d}
 \end{equation*}
is bounded with norm less than or equal to $\sup_{y\in K}(\|y\|^2)$ if and only if \eqref{eq:Bessel} holds. We call the adjoint operator 
 \begin{equation*}
  F^*: L_2(K,\mu) \rightarrow\R^d , \quad f\mapsto \int_{ K } f(x)xd\mu(x)
 \end{equation*}
the \emph{synthesis operator}, where the integral is vector valued. If $\mu \in \mathcal{M}(\mathcal{B},K)$ is a probabilistic frame for $\R^d$ with \emph{frame operator} $S=F^* F$, then $S$ is positive, self-adjoint, and invertible. Moreover, for $\tilde{\mu}=\mu \circ S$, we obtain
\begin{equation}\label{reconsppf}
y= \int_{S^{-1}K}Sz \, \langle z, y\rangle \, d\tilde{\mu}(z)= \int_{S^{-1}K}z \, \langle Sz, y\rangle \, d\tilde{\mu}(z),\quad\text{for all $y\in\R^d$,}
\end{equation} 
which follows from $S^{-1}S=SS^{-1}=\mathcal{I}_d$. In fact, if $\mu \in \mathcal{M}(\mathcal{B},K)$ is a probabilistic frame for $\R^d$, then $\tilde{\mu} \in \mathcal{M}(S^{-1}\mathcal{B},S^{-1}K)$ is a probabilistic frame for $\R^d$. Note that if $\mu$ is the counting measure corresponding to a FUNTF $\{x_i\}_{i=1}^n$, then $\tilde{\mu}$ is the counting measure associated to the canonical dual frame of $\{x_i\}_{i=1}^n$, and Equation \eqref{reconsppf} reduces to \eqref{eq:reconstruction for frames}. These observations motivate the following definition: 
\begin{definition}\label{candualppf} If $\mu \in \mathcal{M}(\mathcal{B},K)$ is a probabilistic frame with frame operator $S$, then $\tilde{\mu}=\mu \circ S \in \mathcal{M}(S^{-1}\mathcal{B},S^{-1}K)$ is  called the \emph{probabilistic canonical dual frame} of $\mu$.
\end{definition}

\begin{remark}
The frame operator $S_1$ of a finite frame $\{x_i\}_{i=1}^n$ has a different normalization than the frame operator $S_2$ of the associated normalized counting measure $\frac{1}{n}\mu_{x_1,\ldots,x_n}$. In fact, we have $S_2=\frac{1}{n}S_1$.
\end{remark}

Next, we generalize finite tight frames: 
\begin{definition}\label{def:prob tight frame}
A probability measure $\mu\in\mathcal{M}(\mathcal{B}, K )$ is called a \emph{probabilistic tight frame} for $\R^d$ if there is a positive constant $0<A$ such that 
 \begin{equation}\label{eq:definition of probabilistic unit norm tight frame}
 A\|x\|^2 = \int_{ K } |\langle x,y\rangle |^2 d\mu (y),\quad\text{for all $x\in\R^d$.}
 \end{equation}
We call $\mu$ a \emph{probabilistic Parseval frame} for $\R^d$ if \eqref{eq:definition of probabilistic unit norm tight frame} holds with $A=1$. The probability measure $\mu$ is called a \emph{probabilistic unit norm tight frame} for $\R^d$ if it is a probabilistic tight frame with $K=S^{d-1}$.
\end{definition}
The following lemma is the probabilistic version of Lemma \ref{lemma:equivalent tight frame conditions discrete} and can be derived from results in continuous frame theory:
\begin{lemma}\label{lemma:equivalent tight frame conditions}
Let $\mu\in\mathcal{M}(\mathcal{B}, K )$ and let $A$ be a positive constant. The following points are equivalent:
\begin{itemize}
\item[\textnormal{(i)}] $\mu$ is a probabilistic tight frame with frame bound $A$,
\item[\textnormal{(ii)}] $F^* F=A\mathcal{I}_d$,
\item[\textnormal{(iii)}] $ x =\frac{1}{A} \int_{ K } \langle x, y\rangle y d\mu(y)$, for all $x\in\R^d$.
\end{itemize}
\end{lemma}

Many properties of finite frames can be carried over. For instance, we can follow the lines in \cite{Christensen:2003aa} to derive a generalization of Lemma \ref{lemma:Parseval S1/2}:
\begin{proposition}
If $\mu\in \mathcal{M}(\mathcal{B},K)$ is a probabilistic frame for $\R^d$, then $\mu\circ S^{1/2}\in \mathcal{M}(S^{-1/2}\mathcal{B},S^{-1/2}K)$ is a probabilistic Parseval frame for $\R^d$. 
\end{proposition}

Only the frame operator and associated operators are used in the proof of Theorem \ref{theorem:fundamental identity} in \cite{Balan:2007aa}. Therefore, we can follow those lines and obtain the fundamental identity and inequality of probabilistic Parseval frames:
\begin{proposition}
Let $\mu\in\mathcal{M}\mathcal{B},K)$ be a probabilistic Parseval frame for $\R^d$. For every measurable subset $J\subset K$ and every $x\in\R^d$, we have
\begin{align*}
\int_J |\langle x,y\rangle |^2d\mu(y) - \big\| \int_J\langle x,y\rangle y d\mu(y)     \big\|^2 & = \int_{K\setminus J} |\langle x,y\rangle |^2d\mu(y) - \big\| \int_{K\setminus J}\langle x,y\rangle y d\mu(y)     \big\|^2,\\
\int_J |\langle x,y\rangle |^2d\mu(y) - \big\| \int_{K\setminus J}\langle x,y\rangle y d\mu(y)     \big\|^2 & \geq \frac{3}{4}\|x\|^2.
\end{align*}
\end{proposition}
If $\{x_i\}_{i=1}^n \subset \R^d$ are pairwise distinct vectors, that form a finite tight frame for $\R^d$, then the normalized counting measure $\frac{1}{n}\mu_{x_1,\ldots,x_n}$ is a probabilistic tight frame for $\R^d$. Lemma \ref{lemma:FUNTF frame bound} can also be carried over: 
\begin{lemma}\label{lemma:frame bounds}
If $\mu\in\mathcal{M}(\mathcal{B}, K )$ is a probabilistic tight frame for $\R^d$, then the frame bound $A$ equals $\frac{1}{d}\int_K \|x\|^2d\mu(x)$.
  \end{lemma}
  \begin{proof}
If $e_1,\ldots,e_d$ is the canonical basis for $\R^d$, then we have $A d = \sum_{i=1}^d A \|e_i\|^2 $. The equality \eqref{eq:definition of probabilistic unit norm tight frame} and finally the Parseval equality for orthonormal bases yield  
  \begin{equation*}
  A d   = \sum_{i=1}^d \int_{K} |\langle e_j,x\rangle |^2 d\mu(x)
    =  \int_{K} \sum_{i=1}^d |\langle e_j,x\rangle |^2 d\mu(x) = \int_{K}\|x\|^2 d\mu(x). \qedhere
    \end{equation*}
  \end{proof}

For $\mu\in\mathcal{M}(\mathcal{B}, K )$, one easily verifies that the frame operator $S=F^* F$ is given by
\begin{equation*}
F^*F:\R^d\rightarrow \R^d,\qquad F^* F (x) = \int_{K} \langle x,y\rangle y d\mu(y).
\end{equation*}
If $\{e_i\}^d$ is the canonical basis for $\R^d$, then the vector valued integral yields
 \begin{equation*}
 \int_{K}  y^{(i)} y d\mu(y)  = \sum_{j=1}^d \int_{K}  y^{(i)} y^{(j)} d\mu(y) e_j,
 \end{equation*}
 where $y=(y^{(1)},\ldots,y^{(d)})^\top\in\R^d$. 
If we denote the second moments of $\mu$ by $m_{i,j}(\mu)$, i.e.,
\begin{equation*}
m_{i,j}(\mu) = \int_K x^{(i)} x^{(j)} d\mu(x),\quad \text{for $i,j=1,\ldots,d$,}
\end{equation*}
then we obtain
 \begin{equation*}
 F^* F e_i = \int_{K}  y^{(i)} y d\mu(y) =  \sum_{j=1}^d\int_{K}  y^{(i)} y^{(j)} d\mu(y) e_j = \sum_{j=1}^d m_{i,j}(\mu) e_j.
 \end{equation*}
Thus, the frame operator is the matrix of second moments. As a consequence, Lemma \ref{lemma:equivalent tight frame conditions} implies the following characterization of probabilistic tight frames:
\begin{corollary}\label{corollary:moments}
A probability measure $\mu\in \mathcal{M}(\mathcal{B},K)$ is a probabilistic tight frame for $\R^d$ if and only if its second moments satisfy 
\begin{equation}\label{eq:second eq moment}
m_{i,j}(\mu)= \frac{1}{d}\delta_{i,j}\int_K \|x\|^2d\mu(x), \quad\text{for all $i,j=1,\ldots,d$.}
\end{equation}
 \end{corollary}

\begin{remark}
Bourgain raised in \cite{Bourgain:1986aa} the following question: \emph{Is there a universal constant $c>0$ such that for any dimension $d$ and any convex body $K$ in $\R^d$ with $\vol_d(K)=1$, there exists a hyperplane $H\subset\R^d$ for which $\vol_{d-1}(K \cap H)>c$}? The positive answer to this question has become known as the hyperplane conjecture. By applying results in \cite{Milman:1987aa}, we can rephrase this conjecture by means of probabilistic tight frames: \emph{There is a universal constant $C$ such that for any convex body $K$, on which the uniform probability measure $\sigma_K$ forms a probabilistic tight frame, the probabilistic tight frame bound is less than $C$}. Due to Lemma \ref{lemma:frame bounds}, the boundedness condition is equivalent to $\int_K \|x\|^2d\sigma_K(x)\leq C d$. The hyperplane conjecture is still open, but there are large classes of convex bodies, for instance, gaussian random polytopes \cite{B.Klartag:2009aa}, for which an affirmative answer has been established. 
\end{remark}

Let us further investigate the uniform probability measure:
\begin{proposition}\label{proposition:uniform distribution}
The uniform probability measure $\sigma_r$ on the sphere of radius $r>0$ is a probabilistic tight frame. 
\end{proposition}
\begin{proof}
We aim to verify the conditions in Corollary \ref{corollary:moments}. First, we consider $i\neq j$: we divide the sphere $S_r^{d-1}=\{x\in\R^d : \|x\|=r\}$ into four parts,
 \begin{align*}
 P_1 & = \{x\in S_r^{d-1}:0\leq x^{(i)},x^{(j)}\leq 1\},\\
 P_2 & = \{x\in S_r^{d-1}:0\leq x^{(i)},-x^{(j)}\leq 1\},\\
 P_3 & = \{x\in S_r^{d-1}:0\leq -x^{(i)},x^{(j)}\leq 1\},\\
 P_4 & = \{x\in S_r^{d-1}:0\leq -x^{(i)},-x^{(j)}\leq 1\}.
 \end{align*}  
Due to symmetry, we obtain 
\begin{equation*}
\int_{P_1}x^{(i)}x^{(j)}d\sigma_r(x) = -\int_{P_2}x^{(i)}x^{(j)}d\sigma_r(x)=-\int_{P_3}x^{(i)}x^{(j)}d\sigma_r(x) = \int_{P_4}x^{(i)}x^{(j)}d\sigma_r(x).
\end{equation*}
Therefore, we derive
 \begin{equation*}
 \int_{S_r^{d-1}}x^{(i)}x^{(j)}d\sigma_r(x) = \sum_{k=1}^4\int_{P_k}x^{(i)}x^{(j)}d\sigma_r(x) = 0.
 \end{equation*}
 
 To tackle $i=j$, we first observe that
 \begin{equation*}
 1=\sigma_r(S_r^{d-1})=\frac{1}{r^2}\int_{S_r^{d-1}}\|x\|^2d\sigma_r(x) = \frac{1}{r^2}\sum_{i=1}^d\int_{S_r^{d-1}}x^{(i)}x^{(i)}d\sigma_r(x). 
  \end{equation*}
 Due to symmetry, the term $\int_{S_r^{d-1}}x^{(i)}x^{(i)} d\sigma_r(x)$ does not depend on the choice of $i$ and we must therefore have $\int_{S_r^{d-1}}x^{(i)}x^{(i)} d\sigma_r(x)=r^2/d$. According to Corollary \ref{corollary:moments}, $\sigma_r$ is a probabilistic tight frame for $\R^d$. 
\end{proof}
\begin{remark}\label{remark:after uniform}
The above proof primarily uses the symmetry of the sphere. Thus, Proposition \ref{proposition:uniform distribution} holds for a much larger class of uniform probability measures on symmetric sets $K$. For instance, it holds for the uniform probability measure on $B_p(r):=\{x\in\R^d : \|x\|_{\ell_p}\leq r\}$ and $\partial B_p(r)$, for $0<p\leq\infty$. 
\end{remark}

Next, we construct continuous nonuniform probability measures on the unit circle that form probabilistic unit norm tight frames. Let $\sigma\in \mathcal{M}(\mathcal{B},S^{d-1})$ represent the uniform probability measure on the circle:
\begin{proposition}\label{theorem:von Mises Fisher}
If $\{x_i\}_{i=1}^n\subset S^1$ is a FUNTF and $f:\R\rightarrow \R$ is a function, such that, for all $i=1,\ldots,n$, $y\mapsto f(\langle x_i,y\rangle )$ is measurable and $\int_{S^1} f(\langle x_i,y\rangle) d\sigma(y) =1$, then the probability measure  
\begin{equation}\label{eq:mix general}
\mu(x) = \frac{1}{n}\sum_{i=1}^n  f(\langle x_i , x\rangle) \sigma(x)
\end{equation}
is a probabilistic unit norm tight frame for $\R^2$.
\end{proposition}
\begin{proof}
Let $\{x_i\}_{i=1}^n=\big\{\big(\begin{smallmatrix}\cos(\alpha_i) \\ \sin(\alpha_i)  \end{smallmatrix}\big) : i=1,\ldots,n\big\}$, for $0\leq \alpha_1,\ldots,\alpha_n<2\pi$. Since, for any $\beta\in[0,2\pi)$, the collection $\big\{\big(\begin{smallmatrix}\cos(\alpha_i+\beta) \\ \sin(\alpha_i+\beta)  \end{smallmatrix}\big) : i=1,\ldots,n\big\}$ is a rotation of $\{x_i\}_{i=1}^n$, it also forms a FUNTF, for all $0\leq \beta \leq 2\pi$. If we parametrize the circle by $[0,2\pi)$, then the mixture in \eqref{eq:mix general} can be carried over to $[0,2\pi)$ and may be written as 
\begin{equation*}
\frac{1}{n} \sum_{i=1}^n f({ \cos(\beta-\alpha_i)}) d\beta,\quad \beta\in [0,2\pi),
\end{equation*}
where we have used $x_i=\big(\begin{smallmatrix} \cos(\alpha_i) \\ \sin(\alpha_i)\end{smallmatrix}\big)$ and $\cos(\beta)\cos(\alpha_i)+\sin(\beta)\sin(\alpha_i)=\cos(\beta-\alpha_i)$. This yields, for any $x\in\R^2$,
\begin{align*}
\int_{S^1}  |\langle y, x\rangle |^2 d\mu(x)
& = 
\int_{0}^{2\pi}  \big|\big\langle x, \big(\begin{smallmatrix} \cos(\beta)\\ \sin(\beta)\end{smallmatrix}\big)\big\rangle \big|^2 \frac{1}{n}\sum_{i=1}^n  f({\cos(\beta-\alpha_i))}d\beta \\
& =
\frac{1}{n}\sum_{i=1}^n \int_{0}^{2\pi}  \big|\big\langle x, \big(\begin{smallmatrix} \cos(\alpha_i+\beta)\\ \sin(\alpha_i+\beta)\end{smallmatrix}\big)\big\rangle \big|^2 f({\cos(\alpha_i+\beta-\alpha_i)})d\beta\\
& = 
 \frac{1}{n}\int_{0}^{2\pi} f({\cos(\beta)}) \sum_{i=1}^n \big|\big\langle x, \big(\begin{smallmatrix} \cos(\alpha_i+\beta)\\ \sin(\alpha_i+\beta)\end{smallmatrix}\big)\big\rangle\big |^2 d\beta\\
&=
\frac{1}{n}\int_{0}^{2\pi}  f({\cos(\beta)}) \frac{n}{2}d\beta =\frac{1}{2}.\qedhere
\end{align*}
\end{proof}

Next, we give an example of Proposition \ref{theorem:von Mises Fisher} that is used in \cite{Ehler:2010ac} to model the patterns found in granular rod experiments: 
\begin{example}
Let $x_0\in S^{1}$ and $\kappa>0$. For the density $f_1(t)=\frac{1}{c_1}\exp(\kappa t)$, we call 
\begin{equation*}
\mu_1(x) = f_1(\langle x_0, x\rangle)\sigma(x)
\end{equation*}
the \emph{von Mises measure}, which reflects the normal distribution on the circle, see \cite{Mardia:2008aa}. The constant $c$ normalizes $\mu$ such that $\mu(S^{1})=1$. The \emph{Watson measure} $\mu\in\mathcal{M}(\mathcal{B},S^{1})$ is given by
\begin{equation*}
\mu_2(x) = f_2(\langle x_0, x\rangle^2)\sigma(x),
\end{equation*}
where $f_2(t)=\frac{1}{c_2}\exp(\kappa t^2)$, and $c_2$ is a normalizing constant, cf.~\cite{Mardia:2008aa}. For $\kappa > 0$, the density of the Watson measure tends
to concentrate around $\pm x_0$, whereas for $\kappa < 0$, the density concentrates around the great circle orthogonal to $x_0$. And as  $|\kappa|$ increases, the density peaks tighten. 

Watson and von Mises measures are widely used in directional statistics. Both densities $f_1$ and $f_2$ satisfy the assumptions of Proposition \ref{theorem:von Mises Fisher}. Therefore, FUNTF mixtures of von Mises and Watson measures according to \eqref{eq:mix general} form probabilistic unit norm tight frames for $\R^2$. 
\end{example}

The proof of Proposition \ref{theorem:von Mises Fisher} implicitly relies on the commutativity of the rotation group in $\R^2$. The special group in $\R^d$, for $d>2$, is not abelian, and we need slightly stronger assumptions. Let $G$ be a finite subgroup of the orthogonal matrices $O(\R^d)$. The \emph{$G$-orbit} of $x\in\R^d$ is the collection $\{ gx : g\in G\}$. The finite subgroup $G$ is called \emph{irreducible} if the $G$-orbit of any nonzero $x\in\R^d$ spans $\R^d$.  If $G\subset O(\R^d)$ is an irreducible finite group, then the $G$-orbit of any nonzero $x\in\R^d$ is a finite tight frame for $\R^d$, cf.~\cite{Vale:2005aa}. The latter can be used to verify that the $n$-th roots of unity, vertices of the platonic solids, and vertices of the truncated icosahedron are finite tight frames, cf.~\cite{Vale:2005aa}. This construction can also be applied to probability distributions:
\begin{proposition}\label{theorem:von Mises Fisher II}
Let $G$ be a finite irreducible subgroup of $O(\R^d)$ and $x_0\in S^{d-1}$. If $\mu\in \mathcal{M}(\mathcal{B},S^{d-1})$, then the probability measure  
\begin{equation*}
\tilde{\mu}(x) = \frac{1}{|G|}\sum_{g\in G} \mu(g^*x)
\end{equation*}
is a probabilistic unit norm tight frame for $\R^d$.
\end{proposition}
\begin{proof}
Since $\{ gx : g\in G\}$ is a finite tight frame, we obtain
\begin{align*}
\int_{S^{d-1}} |\langle y,x\rangle |^2 d \tilde{\mu} (x) & = \frac{1}{|G|}\sum_{g\in G} \int_{S^{d-1}} |\langle y, x\rangle |^2 d\mu (g^* x) \\
& = \frac{1}{|G|}\sum_{g\in G} \int_{S^{d-1}} |\langle y, g x\rangle |^2 d\mu (x) \\
& = \int_{S^{d-1}} \frac{1}{d}\|y\|^2  d\mu (x)=\frac{1}{d}\|y\|^2.\qedhere
\end{align*}
\end{proof}

\subsection{Random Tight Frames for $\R^d$}\label{section:sampling}
The following theorem is the main result of the present paper. Compared to Theorem \ref{theorem:Goyal}, we can replace the uniform distribution with any probabilistic tight frame and the points do not have to be identically distributed. 
To properly formulate the result, let us recall some notation that we already used in Theorem \ref{theorem:Goyal}. We define $E(Z):=\int_{K} Z(x) d\mu(x)$, where $Z:K\rightarrow \R^{p\times q}$ is a random matrix/vector that is distributed according to $\mu\in\mathcal{M}(\mathcal{B},K)$. For notational convenience, we write $Z\in K$ if $Z$ maps into $K$:
\begin{theorem}\label{theorem:final main result}
Let $\{X_{k}\}_{k=1}^n\subset K$ be a collection of random vectors, independently distributed according to probabilistic tight frames $\{\mu_{k}\}_{k=1}^n\subset \mathcal{M}(\mathcal{B},K)$, respectively, whose $4$-th moments are finite, i.e., $N_k:= \int_{K} \|y\|^4 d\mu_{k}(y)<\infty$. If $F$ denotes the random matrix associated to the analysis operator of $\{X_{k}\}_{k=1}^n$, then we have
\begin{equation}\label{eq:theorem final}
E(\|\frac{1}{n}F^* F-\frac{L}{d}\mathcal{I}_d\|_\mathcal{F}^2) =  \frac{1}{n}\big(N-\frac{\tilde{L}}{d}\big),
\end{equation}
where $L:=\frac{1}{n}\sum_{k=1}^n L_{k}$, $\tilde{L}:=\frac{1}{n}\sum_{k=1}^n L_{k}^2$, $L_{k}:=\int_K \|y\|^2d\mu_{k}(y)$, and $N=\frac{1}{n}\sum_{k=1}^n N_k$.
\end{theorem}
Note that Tyler used FUNTFs to derive $M$-estimators of multivariate scatter in \cite{Kent:1988kx,Tyler:1987fk,Tyler:1987uq}. Those results are related to the estimation of the population covariance matrix from the sample covariance, and the latter is closely related to Theorem \ref{theorem:final main result}.  
\begin{proof}
We observe that the $(i,j)$-th entry of the random matrix operator $F^*F$ is given by
\begin{equation*}
(F^*F)_{i,j} = \sum_{k=1}^n X^{(i)}_{k} X^{(j)}_{k},
\end{equation*}
where $X_{k}=(X^{(1)}_{k},\ldots,X^{(d)}_{k})^\top$. First, we fix $(i,j)$ and derive
\begin{align}
E(((\frac{1}{n}F^*F)_{i,j}-\frac{L}{d}\delta_{i,j} )^2 ) & = E(\frac{1}{n^2}\sum_{k,l} X^{(i)}_{k} X^{(j)}_{k}X^{(i)}_{l}X^{(j)}_{l} - \frac{2L}{d}\delta_{i,j}  \frac{1}{n}\sum_{k=1}^n X^{(i)}_{k} X^{(j)}_{k}+\frac{L^2}{d^2}\delta_{i,j} )\nonumber\\
& = \frac{1}{n^2}\sum_{k=1}^n E(X^{(i)}_{k} X^{(j)}_{k}X^{(i)}_{k}X^{(j)}_{k}) +\frac{1}{n^2}\sum_{k\neq l}E(X^{(i)}_{k} X^{(j)}_{k}X^{(i)}_{l}X^{(j)}_{l} )\nonumber\\
& \qquad -\frac{2L}{d}\delta_{i,j}\frac{1}{n}\sum_{k=1}^n E(X^{(i)}_{k} X^{(j)}_{k}) + \frac{L^2}{d^2}\delta_{i,j}.\nonumber
\end{align}
Let us denote $M_k(i,j):=\int_K |y^{(i)}|^2 |y^{(j)}|^2d\mu_k(y)$ and $M=\frac{1}{n}\sum_{k=1}^n M_k$. Since the random vectors are independent and the measures $\{\mu_{k}\}_{k=1}^n$ satisfy \eqref{eq:second eq moment}, we obtain
\begin{align}
E(((\frac{1}{n}F^*F)_{i,j}-\frac{L}{d}\delta_{i,j} )^2 ) & = \frac{1}{n^2}\sum_{k=1}^n M_k(i,j) +\frac{1}{n^2}\sum_{k\neq l} \frac{L_{k}}{d} \frac{L_{l}}{d}\delta_{i,j}  -\frac{2L}{d}\frac{L}{d}   +\frac{L^2}{d^2}\delta_{i,j}\nonumber \\
& = \frac{1}{n^2}\sum_{k=1}^n M_k(i,j) + \frac{1}{n d^2}\sum_{k=1}^n L_{k}(L-\frac{1}{n}L_{k}) \delta_{i,j} -\frac{L^2}{d^2}\delta_{i,j}\nonumber \\
& = \frac{1}{n^2}\sum_{k=1}^n M_k(i,j) + \frac{L^2}{d^2}\delta_{i,j}-\frac{\tilde{L}}{n d^2}\delta_{i,j} -\frac{L^2}{d^2}\delta_{i,j}\nonumber \\
& = \frac{1}{n^2}\sum_{k=1}^n M_k(i,j) - \frac{\tilde{L}}{n d^2}\delta_{i,j} \nonumber\\
& = \frac{1}{n}(M_{i,j} - \frac{\tilde{L}}{d^2}\delta_{i,j}). \label{eq:convergence rate}
\end{align}
The Frobenius norm $\|A\|_{\mathcal{F}}$ of a matrix $A=(a_{i,j})_{i,j}$ equals $\big(\sum_{i,j}a^2_{i,j}\big)^{1/2}$. Since $\frac{1}{n}F^* F-\frac{L}{d}\mathcal{I}_d$ is a $d\times d$ matrix, we obtain 
\begin{equation*}
E(\|\frac{1}{n}F^* F-\frac{L}{d}\mathcal{I}_d\|_\mathcal{F}^2) = \frac{1}{n}(\frac{1}{n}\sum_{k=1}^n \int_K \|y\|^4 d\mu_k(y) - \frac{\tilde{L}}{d}). \qedhere
\end{equation*}
\end{proof}
If the $N_k$ in Theorem \ref{theorem:final main result} are bounded by a universal constant, then \eqref{eq:theorem final} essentially decays as $\frac{1}{n}$. The smaller $N$ the faster tends \eqref{eq:theorem final} to zero. In other words, the $4$-th moments specify the exact decay. 

Let us present few examples that lead to asymptotic tight frames:
\begin{example}
We have already pointed out in Remark \ref{remark:after uniform} that uniform probability measures on $\ell_p$-balls and $\ell_p$-spheres, for $0<p\leq \infty$,  form probabilistic tight frames. According to Theorem \ref{theorem:final main result}, i.i.d.~random points according to the latter distributions approximate a tight frame. 
\end{example}

\begin{remark}
Vershynin has derived a result about the approximation of covariance matrices that is similar to Theorem \ref{theorem:final main result}. His statement is about convergence with high probability in the operator norm. The approximation error is then estimated by a constant times $(\frac{1}{n})^{1/2-2/q}$, where all $\{\mu_{k}\}_{k=1}^n$ must have finite $q$-th moments and $q>4$, cf.~Theorem 6.1 in \cite{Vershynin:2010aa}. For sub-Gaussian distributions, i.e., for $\mu$ such that, for some $s>0$, 
\begin{equation*}
\mu(|\langle X,x\rangle|>t)\leq 2e^{-\frac{t^2}{s^2}},\quad\text{for $t>0$ and $x\in S^{d-1}$,}
\end{equation*}
where $X$ is distributed according to $\mu$, Vershynin can estimate the approximation error in the operator norm by a constant times $(\frac{1}{n})^{1/2}$, cf.~Proposition 2.1 in \cite{Vershynin:2010aa}. Note that the latter matches our decay rates for the mean squared error (we squared the Frobenius norm). Nevertheless, our results address more general distributions since Theorem \ref{theorem:final main result} only requires that the $4$-th moments exist. We do not have any assumption on higher moments, and we do not require that the distributions are sub-Gaussian. 
\end{remark}

For probabilistic unit norm tight frames, Theorem \ref{theorem:final main result} simplifies as follows:
\begin{corollary}\label{corollary:final main result}
Let $\{X_{k}\}_{k=1}^n\subset S^{d-1}$ be a collection of random vectors, independently distributed according to probabilistic unit norm tight frames $\{\mu_{k}\}_{k=1}^n\subset \mathcal{M}(\mathcal{B},S^{d-1})$, respectively. If $F$ denotes the random matrix associated to the analysis operator of $\{X_{k}\}_{k=1}^n$, then 
\begin{equation}\label{eq:in last theorem}
E(\|\frac{1}{n}F^* F-\frac{1}{d}\mathcal{I}_d\|_\mathcal{F}^2) =  \frac{1}{n}\big( 1-\frac{1}{d}\big).
\end{equation}
\end{corollary}
Randomness is used in compressed sensing to design suitable measurements matrices. Each row of such random matrices is a random vector whose covariance must usually be close to the identity matrix. The construction of  random vectors in compressed sensing is commonly based on Bernoulli, Gaussian, and sub-Gaussian distributions. We shall explain that these random vectors are induced by probabilistic tight frames, and in fact, we can apply Theorem \ref{theorem:final main result}: 
\begin{example}\label{example:labeling}
Let $\{X_{k}\}_{k=1}^n$ be a collection of $d$-dimensional random vectors such that each vector's entries are i.i.d~according to a probability measure with zero mean and finite $4$-th moments. This implies that each $X_{k}$ is distributed with respect to a probabilistic tight frame whose $4$-th moments exist. Thus, the assumptions in Theorem \ref{theorem:final main result} are satisfied, and we can compute \eqref{eq:theorem final} for some specific distributions that are related to compressed sensing:
\begin{itemize}
\item If the entries of $X_{k}$, $k=1,\ldots,n$, are i.i.d.~according to a Bernoulli distribution that takes the values $\pm \frac{1}{\sqrt{d}}$ with probability $\frac{1}{2}$, then $X_{k}$ is distributed according to a normalized counting measure supported on the vertices of the $d$-dimensional hypercube. Thus, $X_{k}$ is distributed according to a probabilistic unit norm tight frame for $\R^d$, cf.~Remark \ref{remark:after uniform}, and Corollary \ref{corollary:final main result} can be applied.
\item If the entries of $X_{k}$, $k=1,\ldots,n$, are i.i.d.~according to a Gaussian distribution with $0$ mean and variance $\frac{1}{\sqrt{d}}$, then $X_{k}$ is distributed according to a multivariate Gaussian probability measure $\mu\in\mathcal{M}(\mathcal{B},\R^d)$ whose covariance matrix is $\frac{1}{d}\mathcal{I}_d$, and $\mu$ forms a probabilistic tight frame for $\R^d$. Since the moments of a multivariate Gaussian random vector are well-known, we can explicitly compute $N=1+\frac{2}{d}$, $L=1$, and $\tilde{L}=1$ in Theorem \ref{theorem:final main result}. Thus, the right-hand side of \eqref{eq:theorem final} equals $\frac{1}{n}(1+\frac{1}{d})$. 
\item If the entries of $X_{k}$, $k=1,\ldots,n$, are i.i.d.~with respect to a sub-Gaussian probability measure with $0$ mean, then $X_k$ is distributed according to a probabilistic tight frame for $\R^d$ that has finite moments, and Theorem \ref{theorem:final main result} can be applied.

\end{itemize}
\end{example}





\begin{remark}
When compressed sensing is applied to MRI, the rows of the discrete Fourier matrix $W=\big(\frac{\omega^{j k}}{\sqrt{d}}\big)_{j,k=0}^{d-1}$, where $\omega=e^{\frac{-2\pi \textnormal{i}}{d}}$ and $\textnormal{i}^2=-1$, are usually subsampled to reduce acquisition time. A uniform subsampling of the discrete Fourier matrix is induced by a (complex) probabilistic tight frame: The entire machinery of probabilistic frames for $\R^d$ developed in Section \ref{section:probabilistic unit norm tight frames} can be extended to probabilistic frames for $\C^d$ in a straight-forward manner. Synthesis, analysis, and frame operator can be analogously defined, and a probability measure $\mu$ on $K\subset \C^d$ is then a probabilistic tight frame for $\C^d$ if and only if its ``second moments'' satisfy
\begin{equation*}
\int_K z^{(i)} \overline{z^{(j)}} d\mu(z) = \frac{1}{d} \delta_{i,j}\int_K \|z\|^2 d\mu(z).
\end{equation*} 
Let $\{Z_{k}\}_{k=1}^n$ be a collection of random vectors that are i.i.d.~according to a normalized counting measure $\mu$ supported on the row vectors of the discrete Fourier matrix. Since $W$ is unitary and the absolute value of each entry is $\frac{1}{\sqrt{d}}$, the latter measure is a probabilistic tight frame for $\C^d$, and its ``$4$-th moments'' satisfy $\int_K |z^{(i)}|^2 |z^{(j)}|^2 d\mu(z) = \frac{1}{d^2}$. Corollary \ref{corollary:final main result} can also be extended to probabilistic tight frames for $\C^d$. 
\end{remark}

We conclude this section by rephrasing Theorem \ref{theorem:final main result} in terms of general probability distributions on $K\subset \R^d$ that are not necessarily tight frames. For a matrix  $U=(u_{i,j})\in \R^{d\times d}$, we denote $\|U\|_1 := \sum_{i,j} | u_{i,j} | $:
\begin{theorem}\label{theorem:last results}
Let $\{X_{k}\}_{k=1}^n\subset K$ be a collection of random vectors that are independently distributed according to probability measures $\{\mu_{k}\}_{k=1}^n\subset \mathcal{M}(\mathcal{B},K)$, respectively, whose $4$-th moments are finite, i.e., $N_k:= \int_{K} \|y\|^4 d\mu_{k}(y)<\infty$. Let $\{S_{k}\}_{k=1}^n$ be the frame operators of $\{\mu_{k}\}_{k=1}^n$, respectively. If $F$ denotes the random matrix associated to the analysis operator of $\{X_{k}\}_{k=1}^n$, then we have
\begin{equation*}
E(\|\frac{1}{n}F^* F-S\|_\mathcal{F}^2) =  \frac{1}{n}\big(N-\frac{\|\tilde{S}\|_1}{d^2}\big),
\end{equation*}
where $S=\frac{1}{n}\sum_{k=1}^n S_{k}$, $\tilde{S}_{i,j}=\frac{1}{n}\sum_{k=1}^n ((S_{k})_{i,j})^2$, and $N=\frac{1}{n}\sum_{k=1}^n N_k$. 
\end{theorem}
For instance, Theorem \ref{theorem:last results} applies to random vectors that have a multivariate sub-Gaussian distribution and whose entries are not  necessarily independent. The proof can be derived by following the lines of the proof of Theorem \ref{theorem:final main result} while replacing $\frac{L_{k}}{d}$ with $S_{k}$.

\section{The Probabilistic Frame Potential}\label{section:random points}

\subsection{Minimizing the Probabilistic Frame Potential}\label{section:introducing PFP}
The minimizers of the frame potential are the configurations of $n$ points on the sphere that form a FUNTF. What happens if we have to distribute a continuous mass on the sphere $S^{d-1}$ or, more general, on $K\subset \R^d$?

\begin{definition}
For $0\not\in K$ and $\mu\in\mathcal{M}(\mathcal{B},K)$, 
we call 
\begin{equation}\label{eq:pfp in def}
\PFP(\mu) = \frac{\int_{K}\int_{K} |\langle x,y\rangle|^2 d\mu(x)d\mu(y)}{\big(\int_K \|x\|^2 d\mu(x)\big)^2}
\end{equation}
the \emph{probabilistic frame potential} of $\mu$. 
\end{definition}
We easily observe that $\supp(\mu)\neq \{0\}$ if and only if $\int_K \|x\|^2 d\mu(x) \neq 0$. Therefore, $\PFP(\mu)$ in \eqref{eq:pfp in def} is well-defined. 
We aim to characterize the minimizers of the probabilistic frame potential for fixed $K$. In fact, these minimizers are the probabilistic tight frames provided that the latter exist for the particular choice of $K$. The following theorem generalizes Theorem \ref{theorem:Waldron}:
\begin{theorem}\label{theorem:probabilistic frame potential}
If $0\not\in K$ and $\mu\in \mathcal{M}(\mathcal{B},K)$, then 
\begin{equation}\label{eq:estimate porb pot}
\PFP(\mu)\geq \frac{1}{d},
\end{equation}
and equality holds if and only if $\mu$ is a probabilistic tight frame for $\R^d$.
\end{theorem}
\begin{proof}
Let $m_{i,j}(\mu)$ denote the second moments of $\mu$, i.e., 
$
 m_{i,j}(\mu) = \int_{K} x^{(i)}x^{(j)} d\mu(x)$.  
We obtain
 \begin{equation}\label{eq:first constraint}
 \int_{K} \|x\|^2 d\mu(x) =   \sum_{i=1}^d  \int_{K} x^{(i)}x^{(i)} d\mu(x) = \sum_{i=1}^d m_{i,i}(\mu).
 \end{equation}
 The probabilistic frame potential can be written as
 \begin{align*}
 \PFP(\mu)& =  \frac{\int_{K}\int_{K}\sum_{i=1}^d\sum_{j=1}^d  x^{(i)}y^{(i)} x^{(j)} y^{(j)} d\mu(x)d\mu(y)}{\sum_{i=1}^d m_{i,i}(\mu)}\\
  & =\frac{\sum_{i=1}^d\sum_{j=1}^d m^2_{i,j}(\mu)}{\sum_{i=1}^d m_{i,i}(\mu)}.
 \end{align*}
The H{\"o}lder inequality implies
\begin{equation}\label{eq:proof estimate prop pot}
 \sum_{i=1}^d m_{i,i}(\mu) \leq \big(\sum_{i=1}^d m^2_{i,i}(\mu)\big)^{1/2} \big(\sum_{i=1}^d 1\big)^{1/2}\leq \big(\sum_{i=1}^d \sum_{j=1}^d m^2_{i,j}(\mu)\big)^{1/2} d^{1/2},
\end{equation}
which yields \eqref{eq:estimate porb pot}.

Next, assume that the latter inequalities \eqref{eq:proof estimate prop pot}, in fact, are equalities. This requires $m_{i,j}(\mu)=0$, for all $i\neq j$, and the H{\"o}lder inequality was actually an equality. The H{\"o}lder inequality becomes an equality if and only if the occurring sequences are linearly dependent. Thus, $(m_{i,i}(\mu))_{i=1}^d$ must be a multiple of the constant sequence. Due to \eqref{eq:first constraint}, we obtain $m_{i,i}(\mu)=\frac{1}{d} \int_{K} \|x\|^2 d\mu(x)$, for all $i=1,\ldots,d$, and hence $\mu$ is a probabilistic tight frame, cf.~Corollary \ref{corollary:moments}.  

Conversely, if $\mu$ is a probabilistic tight frame, then $m_{i,j}(\mu)=\delta_{i,j}\frac{1}{d} \int_{K} \|x\|^2 d\mu(x)$ due to Corollary \ref{corollary:moments}. Thus, we have equality in \eqref{eq:proof estimate prop pot} and hence in \eqref{eq:estimate porb pot}.
 \end{proof}
 
According to Proposition \ref{proposition:uniform distribution}, probabilistic tight frames exist for $K=S^{d-1}$. If $K=\R^d\setminus \{0\}$, then the normalized counting measure of any finite tight frame is a probabilistic tight frame. Hence, Theorem \ref{theorem:probabilistic frame potential} leads to the following generalization of Theorem \ref{theorem:Benedetto Fickus} and Theorem \ref{theorem:Waldron}: 
 \begin{corollary}\label{corollary:min prob etc}
If $K=S^{d-1}$, then the minimizers of the probabilistic frame potential are exactly the probabilistic unit norm tight frames for $\R^d$. If $K=\R^d\setminus \{0\}$, then the minimizers of the probabilistic frame potential are exactly the probabilistic tight frames for $\R^d$.
\end{corollary}
 
%
%
 
Let us explore the relations between Corollary \ref{corollary:min prob etc} and the discrete frame potential in Theorem \ref{theorem:Benedetto Fickus}. For fixed $d$ and $K=S^{d-1}$, every FUNTF induces a minimizer of the probabilistic frame potential:
\begin{example}
If $\{x_i\}_{i=1}^n\subset S^{d-1}$ is a FUNTF, then $\FP(\{x_i\}_{i=1}^n)=\frac{n^2}{d}$ according to Theorem \ref{theorem:Benedetto Fickus}. Thus, the discrete point measure $\frac{1}{n}\mu_{x_1,\ldots,x_n}$ satisfies $\PFP(\frac{1}{n}\mu_{x_1,\ldots,x_n})= \frac{1}{d}$, and therefore is a minimizer of the probabilistic frame potential for $K=S^{d-1}$. 
\end{example}
Contrary to Theorem \ref{theorem:Benedetto Fickus}, orthonormal systems that are not a basis, do not induce a minimizer:
\begin{example}
Let $\{x_i\}_{i=1}^n\subset S^{d-1}$ be an orthonormal system with $n<d$. Due to Theorem \ref{theorem:Benedetto Fickus}, we have $\FP(\{x_i\}_{i=1}^n)=n$. For $K=S^{d-1}$, this implies $\PFP(\frac{1}{n}\mu_{x_1,\ldots,x_n}) =\frac{n}{n^2}$. Since $n<d$, we deduce $\PFP(\frac{1}{n}\mu_{x_1,\ldots,x_n}) = \frac{1}{n}>\frac{1}{d}$.
\end{example}

\subsection{Relations to Spherical $t$-designs}\label{section:spherical 2-designs}
Let $\sigma$ denote the uniform probability measure on $S^{d-1}$. A \emph{spherical $t$-design} is a finite subset $\{x_i\}_{i=1}^n\subset S^{d-1}$,
such that,
\begin{equation*}
\frac{1}{n}\sum_{i=1}^n h(x_i) = \int_{S^{d-1}} h(x)d\sigma(x),
\end{equation*}
for all homogeneous polynomials $h$ of total degree less than or equal to $t$ in $d$ variables, cf.~\cite{Delsarte:1977aa}. We call a probability measure $\mu\in\mathcal{M}(\mathcal{B},S^{d-1})$ a \emph{probabilistic spherical $t$-design} if
\begin{equation}\label{eq:prob spherical design}
\int_{S^{d-1}}h(x)d\mu(x) = \int_{S^{d-1}} h(x)d\sigma(x),
\end{equation}
for all homogeneous polynomials $h$ with total degree less than or equal to $t$.

\begin{theorem}\label{theorem spherical and FP}
If $\mu \in\mathcal{M}(\mathcal{B},S^{d-1})$, then the following are equivalent:
\begin{itemize}
\item[\textnormal{(i)}] $\mu$ is a probabilistic spherical $2$-design.
\item[\textnormal{(ii)}] $\mu$ minimizes
\begin{equation}\label{mixed potential}
\frac{\int_{S^{d-1}} \int_{S^{d-1}} |\langle x,y\rangle |^2 d\mu(x) d\mu(y)}{\int_{S^{d-1}} \int_{S^{d-1}}  \|x-y\|^2 d\mu(x)d\mu(y)}
\end{equation}
among all probability measures $\mathcal{M}(\mathcal{B},S^{d-1})$. 
\item[\textnormal{(iii)}] $\mu$ satisfies
\begin{align}
 \int_{S^{d-1}}x d\mu(x) \label{mean condition} & = 0\\
\int_{S^{d-1}} x^{(i)} x^{(j)} d\mu(x) & = \frac{1}{d}\delta_{i,j} \label{diagonal condition}.
\end{align}
\end{itemize}
In particular, if $\mu$ is a probabilistic unit norm tight frame, then $\nu(A):=\frac{1}{2}(\mu(A)+\mu(-A))$, for $A\in\mathcal{B}$, defines a probabilistic spherical $2$-design. 
\end{theorem}
\begin{proof}
To show that (i) and (iii) are equivalent, we observe that the uniform probability measure $\sigma$ is a probabilistic unit norm tight frame, cf.~Proposition \ref{proposition:uniform distribution}. It hence satisfies \eqref{diagonal condition} according to Corollary \ref{corollary:moments}. Due to its symmetry, $\sigma$ also satisfies \eqref{mean condition}. Thus according to \eqref{eq:prob spherical design}, the probabilistic spherical $2$-designs are exactly those probability measures $\mu\in\mathcal{M}(\mathcal{B},S^{d-1})$ that satisfy \eqref{mean condition} and \eqref{diagonal condition}.

To address the equivalence between (ii) and (iii), we will observe that the minimization \eqref{mixed potential} splits into minimizing its numerator and maximizing its denominator. Due to Corollary \ref{corollary:moments} and Corollary \ref{corollary:min prob etc}, the numerator is minimized if and only if $\mu$ satisfies \eqref{diagonal condition}. Let us rewrite the denominator as follows:
 \begin{align*}
  \int_{S^{d-1}} \int_{S^{d-1}} \|x-y\|^2 d\mu(x)d\mu(y) & =  \int_{S^{d-1}} \int_{S^{d-1}}\sum_{i=1}^d x^{(i)}x^{(i)}+y^{(i)}y^{(i)}-2x^{(i)}y^{(i)} d\mu(x)d\mu(y) \\
  & = \int_{S^{d-1}} \int_{S^{d-1}}2 d\mu(x)d\mu(y)-2\sum_{i=1}^d \int_{S^{d-1}} \int_{S^{d-1}}x^{(i)}y^{(i)}d\mu(x)d\mu(y)\\
  & = 2-2\sum_{i=1}^d \Big( \int_{S^{d-1}} x^{(i)} d\mu(x)\Big)^2.
 \end{align*}
It is hence maximized if and only if $\int_{S^{d-1}} xd\mu(x) = 0$. Thus, (iii) implies (ii). For the reverse implication, we need to verify that there is a probability measure that minimizes the numerator and maximizes the denominator of \eqref{mixed potential} at the same time. We first recall that  probabilistic unit norm tight frames exist, cf.~Proposition \ref{proposition:uniform distribution}. If $\mu$ is such a probabilistic unit norm tight frame, then $\nu$ as defined in Theorem \ref{theorem spherical and FP} satisfies \eqref{mean condition}, and $\nu$ also satisfies \eqref{diagonal condition} since its second moments coincide with those of $\mu$. Hence, (ii) implies (iii), and we can conclude the proof. 
\end{proof}

\begin{remark}
We have shown in the proof of Theorem \ref{theorem spherical and FP} that the maximizers of $ \int_{S^{d-1}} \int_{S^{d-1}} \|x-y\|^2 d\mu(x)d\mu(y)$ are exactly the zero mean probability measures on the sphere. The latter result is already implicitly contained in a work by Bjoerck \cite{Bjoerck:1955aa}, in which he considers the integrals over the unit ball and then shows that the mass of the maximizer must completely be contained in the unit sphere.  
\end{remark}

\section{Conclusions}\label{section:conclusions}
First, we introduced probabilistic frames and verified that many properties from finite frames can be adopted. Secondly, we used probabilistic tight frames to significantly improve a result by Goyal, Vetterli, and Thao in \cite{Goyal:1998aa} about the random choice of points on the sphere. We still approximate a tight frame while allowing for a much wider class of probability measures, namely any probabilistic tight frame. The requirement of identical distributions is also removed. We also verified that many random matrices, which are used in compressed sensing, are induced by probabilistic tight frames. Thirdly, we extended results about the frame potential as introduced by Benedetto and Fickus in \cite{Benedetto:2003aa}. In fact, we demonstrated that probabilistic tight frames are the minimizers of the probabilistic frame potential.

\section*{Acknowledgements}
The author was supported by the Intramural Research Program of the National Institute of Child Health and Human Development and by NIH/DFG Research Career Transition Awards Program (EH 405/1-1/575910).  
%

\providecommand{\bysame}{\leavevmode\hbox to3em{\hrulefill}\thinspace}
\providecommand{\MR}{\relax\ifhmode\unskip\space\fi MR }
\providecommand{\MRhref}[2]{%
  \href{http://www.ams.org/mathscinet-getitem?mr=#1}{#2}
}
\providecommand{\href}[2]{#2}

 
\end{document}